\documentclass[oneside,english]{amsart}
\usepackage[]{fontenc}
\usepackage[latin9]{inputenc}
\usepackage{color}
\usepackage{babel}
\usepackage{textcomp}
\usepackage{amsthm}
\usepackage{amstext}
\usepackage{amssymb}
\usepackage[numbers]{natbib}
\usepackage[unicode=true,pdfusetitle,
 bookmarks=true,bookmarksnumbered=false,bookmarksopen=false,
 breaklinks=false,pdfborder={0 0 1},backref=false,colorlinks=true]
 {hyperref}
\hypersetup{
 allcolors=blue}
\usepackage{breakurl}

\makeatletter
\numberwithin{equation}{section}
\numberwithin{figure}{section}
\theoremstyle{plain}
\newtheorem{thm}{\protect\theoremname}
  \theoremstyle{plain}
  \newtheorem{lem}[thm]{\protect\lemmaname}
  \theoremstyle{plain}
  \newtheorem{prop}[thm]{\protect\propositionname}
  \theoremstyle{plain}
  \newtheorem{cor}[thm]{\protect\corollaryname}

\allowdisplaybreaks[4]
\usepackage{graphicx}
\usepackage[a4paper]{geometry} 
\let\tempone\itemize
\let\temptwo\enditemize
\renewenvironment{itemize}{\tempone\addtolength{\itemsep}{0.3\baselineskip}}{\temptwo}

\usepackage{listings}
\usepackage{xcolor}
\makeatother

\usepackage{listings}
  \providecommand{\corollaryname}{Corollary}
  \providecommand{\lemmaname}{Lemma}
  \providecommand{\propositionname}{Proposition}
\providecommand{\theoremname}{Theorem}

\begin{document}

\title{A classification of permutation polynomials of degree $7$ over finite fields}

\author{Xiang Fan}

\address{School of Mathematics, Sun Yat-sen University, Guangzhou 510275,
China}
\begin{abstract}
Up to linear transformations, we give a classification of all permutation polynomials
of degree $7$ over $\mathbb{F}_{q}$ for any odd prime power $q$,  with the help of
the SageMath software.
\end{abstract}

\maketitle

\section{Introduction}

Let $\mathbb{F}_{q}$ denote the finite field of order $q=p^{r}$
and characteristic $p$, with the multiplicative group $\mathbb{F}_{q}^{*}=\mathbb{F}_{q}\backslash\{0\}$.
Reserve the letter $x$ for the indeterminate of the polynomial ring
$\mathbb{F}_{q}[x]$ in one variable with coefficients in $\mathbb{F}_{q}$.
Each mapping of $\mathbb{F}_{q}$ into itself can be expressed by
a polynomial in $\mathbb{F}_{q}[x]$. By a \emph{permutation polynomial}
(PP) over $\mathbb{F}_{q}$, we mean a polynomial $f\in\mathbb{F}_{q}[x]$
with the associated mapping $a\mapsto f(a)$ permuting $\mathbb{F}_{q}$.

Arose in Hermite \citep{Hermite1863sur}  and Dickson \citep{Dickson1897analytic},
the theory of PPs over have been a hot topic of study for more than
one hundred years. However, the basic problem of classification or
characterization of PPs of prescribed forms are still challenging.

In Dickson's thesis work \citep{Dickson1897analytic}, a complete
list of normalized PPs of degree $d$ over $\mathbb{F}_{q}$ was given
for $d\leqslant5$ with any $q$, and for $d=6$ with $q$ coprime
to $6$. Dickson \citep{Dickson1897analytic} also classified all
PPs of degree $6$ over $\mathbb{F}_{3^{r}}$ up to linear transformations.
(See also Shallue and Wanless \citep{ShallueWanless2013permutation}
for a reconfirmation of the completeness of Dickson's classification.)
Recently, the classification of all PPs of degree 6 and 7 over $\mathbb{F}_{2^{t}}$
was obtained by Li, Chandler and Xiang \citep{LiChandlerXiang2010permutation}.
The present paper contributes to this line by classifying all PPs
of degree $7$ over $\mathbb{F}_{q}$ with any odd $q$, up to linear
transformations.

Our approach relies heavily on some well-established results about
exceptional polynomials. Here an \emph{exceptional polynomial} over
$\mathbb{F}_{q}$ stands for a polynomial in $\mathbb{F}_{q}[x]$
which is a PP over $\mathbb{F}_{q^{m}}$ for infinitely many $m$.
A notable pattern in Dickson's classification is that PPs of a given
degree are indeed exceptional polynomials except a few ``accidents''
for small $q$. This pattern is intrinsic and true for any degree,
by the following result of von zur Gathen \citep[Theorem 1]{Gathen1991values}
:
\begin{lem}
\label{prop:n4} A non-exceptional PP of degree $n$ over $\mathbb{F}_{q}$
exists only if $q\leqslant n^{4}$.
\end{lem}
Recently, Chahal and Ghorpade \citep[Remark 3.4]{ChahalGhorpade2018Carlitz}
replaced $n^{4}$ in Lemma \ref{prop:n4} to 
\[
\left(\frac{(n-2)(n-3)+\sqrt{(n-2)^{2}(n-3)^{2}+2(n^{2}-1)}}{2}\right)^{2}
\]
which is less than $n^{2}(n-2)^{2}$. Moreover, our preprint \citep{Fan2019Weil}
further refined this bound as follows.
\begin{lem}
\label{lem:bound} A non-exceptional PP of degree $n$ over $\mathbb{F}_{q}$
exists only if 
\[
q+1\leqslant(n-2)(n-3)\dfrac{\lfloor2\sqrt{q}\rfloor}{2}+2(n-1),
\]
and in particular $q\leqslant\left\lfloor \left(\dfrac{(n-2)(n-3)+\sqrt{(n-2)^{2}(n-3)^{2}+8n-12}}{2}\right)^{2}\right\rfloor $.
As a corollary, a non-exceptional PP of degree $7$ over $\mathbb{F}_{q}$
exists only if $q\leqslant409$.
\end{lem}
Up to linear transformations, all exceptional polynomials of degree
7 over $\mathbb{F}_{q}$ are classified by the following results of
M\"{u}ller \citep{Muller1997Weil}, and Fried, Guralnick, Saxl \citep{FriedGuralnickSaxl1993Schur}.
Let us say that two polynomials $f$ and $g$ in $\mathbb{F}_{q}[x]$
are \emph{related by linear transformations}, or \emph{linearly related}
for short, if there exists $s,t\in\mathbb{F}_{q}^{*}$ and $u,v\in\mathbb{F}_{q}$
such that $g(x)=sf(tx+u)+v$. The linearly related $f$ and $g$ share
the same degree, and $f$ is a PP (resp. exceptional polynomial) over
$\mathbb{F}_{q}$ if and only if so is $g$. 
\begin{lem}
\emph{\citep[Theorem 4]{Muller1997Weil}} Every exceptional polynomial
$f$ of degree $n$ over $\mathbb{F}_{q}$ with $\mathrm{gcd}(q,n)=1$
is a composition of Dickson polynomials and linear polynomials. In
particular, when $n=7$, $f$ is linearly related to 
\[
D_{7}(x,a):=x^{7}-7ax^{5}+14a^{2}x^{3}-7a^{3}x
\]
with either $a=0$ and $q\not\equiv1$ $(\mathrm{mod}\ 7)$, or $a\in\mathbb{F}_{q}^{*}$
and $q^{2}\not\equiv1$ $(\mathrm{mod}\ 7)$.
\end{lem}

\begin{lem}
\emph{\label{lem:EP7p7}\citep[Theorem 8.1]{FriedGuralnickSaxl1993Schur}}
For a prime $p$, each exceptional polynomial $f$ of degree $p$
over $\mathbb{F}_{p^{r}}$ is linearly related toa polynomial $x(x^{\frac{p-1}{s}}-a)^{s}$
with $s\mid(p-1)$ and $a\in\mathbb{F}_{p^{r}}$ such that $a^{\frac{s(p^{r}-1)}{p-1}}\neq1$.
In particular, when $p=7$, $f$ is linearly related to $x^{7}$ or
to $x(x^{\frac{6}{s}}-a)^{s}$ with $s\in\{1,2,3\}$ and $a\in\mathbb{F}_{7^{r}}^{*}$
such that $a^{\frac{s(7^{r}-1)}{6}}\neq1$. 
\end{lem}
Therefore, to complete the list of PPs of degree $7$, we only need
to consider the non-exceptional PPs over $\mathbb{F}_{q}$ with $8\leqslant q\leqslant409$.
In principle this can be exhausted by running a computer program.
Section \ref{sec:Result} will realize this by SageMath \citep{sagemath},
a free open-source mathematics software system based on Python and
many open-source packages. We omit the cases of $q=2^{r}$, which
are covered the following lemma quoted from \citep{LiChandlerXiang2010permutation}.
\begin{lem}
\emph{\citep[Theorem 4.4]{LiChandlerXiang2010permutation}} Let $3\leqslant r\in\mathbb{Z}$.
A PP over $\mathbb{F}_{2^{r}}$ exists only if $3\nmid r$. When $3\nmid r$,
every exceptional polynomial over $\mathbb{F}_{2^{r}}$ is linearly
related to $x^{7}$ or $x^{7}+x^{5}+x$. Let $f$ be a non-exceptional
PP over $\mathbb{F}_{2^{r}}$, then $r=4$ and $f$ is linearly related
to one of the following: 
\begin{align*}
x^{7}+x^{4}+x,\quad x^{7}+x^{5}+x^{4}, & \quad x^{7}+x^{5}+ax^{4}+a^{14}x^{3}+a^{12}x^{2}+a^{8}x,\\
x^{7}+x^{5}+a^{7}x^{4}+a^{5}x^{2}+a^{3}x, & \quad x^{7}+x^{5}+a^{5}x^{4}+a^{2}x^{3}+a^{12}x^{2}+a^{5}x,
\end{align*}
with $a$ running over the set $\{a\in\mathbb{F}_{16}:a^{4}+a+1=0\}$.
\end{lem}
Our main results in Section \ref{sec:Result} can be stated as follows.
\begin{thm}
If a non-exceptional PP $f$ of degree $7$ over $\mathbb{F}_{q}$
exists for $q>7$, then  
\[
q\in\{9,11,13,17,19,23,25,27,31,49\}.
\]
Up to linear transformations, $f$ is listed in Theorem \ref{thm:Main},
Proposition \ref{prop:p3} and \ref{prop:p7}.
\end{thm}

\section{\label{sec:Class}Assumptions by linear transformations}

The following parts of this note aim to list all PPs of degree $7$,
up to linear transformations, over $\mathbb{F}_{q}$ with an odd prime
power $q$ such that $7<q\leqslant409$. In this section, we will
show that each PP of degree $7$ is linearly related to some $f(x)=x^{7}+\sum_{i=1}^{5}a_{i}x^{i}$
with coefficients $a_{i}\in\mathbb{F}_{q}$ subject to certain assumptions
in Proposition \ref{prop:Eq-all}.
\begin{prop}
\label{prop:Eq-cop7} Let $q$ be a prime power coprime to $7$. Then
every polynomial of degree $7$ in $\mathbb{F}_{q}[x]$ is linearly
related to some $f\in\mathbb{F}_{q}[x]$ in normalized form, namely
$f(x)=x^{7}+\sum_{i=1}^{5}a_{i}x^{i}$ with all $a_{i}\in\mathbb{F}_{q}$.
Moreover, for any $g(x)=x^{7}+\sum_{i=1}^{5}b_{i}x^{i}\in\mathbb{F}_{q}[x]$
with all $b_{i}\in\mathbb{F}_{q}$, $f$ and $g$ are linearly related
if and only if $f(x)=t^{7}g(t^{-1}x)$ for some $t\in\mathbb{F}_{q}^{*}$,
or equivalently $a_{i}=b_{i}t^{7-i}$ for $1\leqslant i\leqslant5$.\end{prop}
\begin{proof}
Each polynomial $h$ of degree $7$ in $\mathbb{F}_{q}[x]$ can be
written as $h(x)=\sum_{i=1}^{7}c_{i}x^{i}$, with all $c_{i}\in\mathbb{F}_{q}$
and $c_{7}\neq0$. Let $f(x)=c_{7}^{-1}h(x-7^{-1}c_{6})-c_{7}^{-1}h(-7^{-1}c_{6})$,
 then $f$ is in normalized form and linearly related to $h$.

Suppose $f(x)=x^{7}+\sum_{i=1}^{5}a_{i}x^{i}$ and $g(x)=x^{7}+\sum_{i=1}^{5}b_{i}x^{i}$
(with all $a_{i}$, $b_{i}\in\mathbb{F}_{q}$) are linearly related,
namely $g(x)=sf(tx+u)+v$ with $s,t\in\mathbb{F}_{q}^{*}$ and $u,v\in\mathbb{F}_{q}$.
Clearly, $st^{7}=1$ and $7st^{6}u=0$, considering the coeffcients
of $x^{7}$ and $x^{6}$. So $u=0$ and $s=t^{-7}$. Then $g(x)=t^{-7}f(tx)+v$,
and $v=g(0)=0$.
\end{proof}
When $7\mid q$, the definition of normalized form only requires that
$f$ is monic and $f(0)=0$. However, we can choose $f$ in a more
reduced fashion as follows, up to linear transformations.
\begin{prop}
\label{prop:Eq-7} Let $r$ be a positive integer. Then every polynomial
of degree $7$ in $\mathbb{F}_{7^{r}}[x]$ is linearly related to
$x^{7}$ or to some $f\in\mathbb{F}_{7^{r}}[x]$ of the form $f(x)=x^{7}+\sum_{i=1}^{k}a_{i}x^{i}$
with $1\leqslant k\leqslant6$, all $a_{i}\in\mathbb{F}_{7^{r}}$,
and $a_{k}\neq0=a_{k-1}$. Moreover, for any $g(x)=x^{7}+\sum_{i=1}^{k'}b_{i}x^{i}\in\mathbb{F}_{7^{r}}[x]$
with $1\leqslant k'\leqslant6$, all $b_{i}\in\mathbb{F}_{7^{r}}$
and $b_{k'}\neq0=b_{k'-1}$, $f$ and $g$ are linearly related if
and only if $f(x)=t^{7}g(t^{-1}x)$ for some $t\in\mathbb{F}_{7^{r}}^{*}$,
or equivalently $k=k'$ and $a_{i}=b_{i}t^{7-i}$ for $1\leqslant i\leqslant k$.\end{prop}
\begin{proof}
Each polynomial $h$ of degree $7$ in $\mathbb{F}_{7^{r}}[x]$ can
be written as $h(x)=\sum_{i=0}^{7}c_{i}x^{i}$, with all $c_{i}\in\mathbb{F}_{7^{r}}$
and $c_{7}\neq0$. If $c_{j}=0$ for all $1\leqslant j\leqslant6$,
then $h(x)=c_{7}x^{7}+c_{0}$ is linearly related to $x^{7}$. Otherwise,
let $k=\mathrm{max}\{j\in\mathbb{Z}:c_{j}\neq0,1\leqslant j\leqslant6\}$,
then $h(x)=c_{7}x^{7}+\sum_{i=0}^{k}c_{i}x^{i}$ with $c_{k}\neq0$,
and we can pick $f(x)=c_{7}^{-1}h(x-k^{-1}c_{k}^{-1}c_{k-1})-c_{7}^{-1}h(-k^{-1}c_{k}^{-1}c_{k-1})$
to meet the requirement.

For $k,k'\in\{1,2,\dots,6\}$, suppose that $f(x)=x^{7}+\sum_{i=1}^{k}a_{i}x^{i}$
and $g(x)=x^{7}+\sum_{i=1}^{k'}b_{i}x^{i}$ are linearly related,
with all $a_{i},b_{i}\in\mathbb{F}_{7^{r}}$, $a_{k}\neq0$, $b_{k'}\neq0$,
and $a_{k-1}=b_{k'-1}=0$. Let $g(x)=sf(tx+u)+v$ with $s,t\in\mathbb{F}_{7^{r}}^{*}$
and $u,v\in\mathbb{F}_{7^{r}}$. Clearly, $s=t^{-7}$ by the coeffcients
of $x^{7}$. Moreover, $g(x)=x^{7}+t^{k-7}a_{k}x^{k}+\cdots=x^{7}+b_{k'}x^{k'}+\cdots$
with nonzero $a_{k}$ and $b_{k'}$, so $k=k'$ and $t^{k-7}a_{k}=b_{k}$.
If $k=1$, then  $t^{-6}a_{1}=b_{1}$. If $k\geqslant2$, then $skt^{k-1}u=0$
by the coefficients of $x^{k-1}$, and thus $u=0$. Always we have
$g(x)=t^{-7}f(tx)$.\end{proof}
\begin{lem}
\emph{\label{lem:a6zero}\citep[\S65]{Dickson1897analytic}} Let $p$
be a prime, and $r,s$ be positive integers such that $s\mid r$.
If $f(x)$ is a PP of degree $p^{s}$ over $\mathbb{F}_{p^{r}}$,
then the coefficient of $x^{p^{s}-1}$ in $f$ is $0$. In particular,
the coefficient of $x^{6}$ is $0$ in a PP of degree $7$ over $\mathbb{F}_{7^{r}}$.
\end{lem}
Consider the group homomorphism $\theta_{m}:\mathbb{F}_{q}^{*}\to\mathbb{F}_{q}^{*}$
defined by $\theta_{m}(t)=t^{m}$, for $t\in\mathbb{F}_{q}^{*}$ and
$m\in\mathbb{Z}$. Let $\mathrm{CK}_{q}(m)$ be a complete set of
coset representatives of its cokernel $\mathbb{F}_{q}^{*}/\theta_{m}(\mathbb{F}_{q}^{*})$,
and let $\mathrm{CI}_{q}(m)$ be a complete set of coset representatives
of its coimage $\mathbb{F}_{q}^{*}/\mathrm{ker}(\theta_{m})$.
\begin{prop}
\label{prop:Eq-all} Each PP of degree $7$ over $\mathbb{F}_{q}$
is linearly related to $x^{7}$ or to some $f(x)=x^{7}+\sum_{i=1}^{k}a_{i}x^{i}$
with $1\leqslant k\leqslant5$ and all $a_{i}\in\mathbb{F}_{q}$ satisfying
the following requirements:\begin{itemize}

\item $0\neq a_{k}\in\mathrm{CK}_{q}(7-k)$ and $a_{k-1}\in\{0\}\cup\mathrm{CI}_{q}(7-k)$;

\item if $7\mid q$ then $a_{k-1}=0$;

\item if $k=5$ and $a_{4}=0$, then $a_{2}\in\{0\}\cup\mathrm{CI}_{q}(2)$;

\item if $k=4$ and $a_{3}=0$, then $a_{2}\in\{0\}\cup\mathrm{CI}_{q}(3)$;

\item if $k=3$, $a_{2}=0$ and $q\equiv1$ $(\mathrm{mod}\ 4)$,
then $a_{1}\in\{0\}\cup\mathrm{CI}_{q}(2)$.

\end{itemize}
\end{prop}
Let $e$ be a generator of the multiplicative group $\mathbb{F}_{q}^{*}$.
For later use, we can take  
\begin{align*}
\mathrm{CK}_{q}(m) & =\{e^{j}:0\leqslant j<\mathrm{gcd}(m,q-1)\},\\
\mathrm{CI}_{q}(m) & =\{e^{j}:0\leqslant j<(q-1)/\mathrm{gcd}(m,q-1)\}.
\end{align*}

\begin{cor}
\label{cor:EPpnot7}Let $f$ be an exceptional polynomialof degree
$7$ over $\mathbb{F}_{q}$ with $7\nmid q$, and let $e$ be a generator
of $\mathbb{F}_{q}^{*}$.

(1) If $q\equiv6$ $(\mathrm{mod}\ 7)$, then $f$ is linearly related
to $x^{7}$.

(2) If $q\not\equiv\pm1$ $(\mathrm{mod}\ 7)$, then $f$ is linearly
related to exactly one of: $x^{7}$,  
\begin{align*}
D_{7}(x,-7^{-1}) & =x^{7}+x^{5}+2\cdot7^{-1}x^{3}+7^{-2}x,\\
D_{7}(x,-7^{-1}e) & =x^{7}+ex^{5}+2\cdot7^{-1}e^{2}x^{3}+7^{-2}e^{3}x.
\end{align*}

\end{cor}

\section{\label{sec:HC}Equalities by Hermite\textquoteright s criterion}

Let $\mathbb{N}=\{i\in\mathbb{Z}:i\geqslant0\}$. For $i\in\mathbb{N}$
and $f\in\mathbb{F}_{q}[x]$, let $[x^{i}:f]$ denote the coefficient
of $x^{i}$ in $f(x)$. Namely, $f(x)=\sum_{i=0}^{\deg(f)}[x^{i}:f]x^{i}$,
and $[x^{j}:f]=0$ for any $j>\deg(f)$.

The following criterion for PPs, introduced by Hermite \citep{Hermite1863sur}
for prime fields $\mathbb{F}_{p}$ and generalized by Dickson \citep{Dickson1897analytic},
provides equalities for coefficients of PPs over $\mathbb{F}_{q}$.
\begin{lem}
[{Hermite\textquoteright s criterion \citep[\S7.6]{LidlNiederreiter1983finite}}]
A polynomial $f\in\mathbb{F}_{q}[x]$ is a PP if and only if the following
two conditions hold:

(1) $\sum_{j=1}^{\deg(f)}[x^{j(q-1)}:f^{q-1}]\neq0$;

(2) for every integer $k$ coprime to $q$ with $1\leqslant k\leqslant q-2$,
$\sum_{j=1}^{\lfloor\frac{k\deg(f)}{q-1}\rfloor}[x^{j(q-1)}:f^{k}]=0$.
\end{lem}
For integers $k_{1}$, $k_{2}$, $\dots$, $k_{t}$ and $k$, recall
the multinomial coefficient defined as 
\[
\binom{k}{k_{1},k_{2},\dots,k_{t}}:=\begin{cases}
\dfrac{k!}{k_{1}!k_{2}!\cdots k_{t}!} & \text{if all }k_{1},\dots,k_{t}\geqslant0\text{ and }k_{1}+k_{2}+\cdots+k_{t}=k,\\
0 & \text{otherwise}.
\end{cases}
\]
For $f(x)=x^{7}+\sum_{i=1}^{5}a_{i}x^{i}$ with all $a_{i}\in\mathbb{F}_{q}$,
note that 
\[
f(x)^{k}=\sum_{j_{1}+\cdots+j_{5}+j_{7}=k}\binom{k}{j_{7},j_{1},j_{2},\dots,j_{5}}(\prod_{i=1}^{5}a_{i}^{j_{i}})x^{j_{1}+2j_{2}+3j_{3}+4j_{4}+5j_{5}+7j_{7}}.
\]
Consider $(j_{1},\dots,j_{5},j_{7})\in\mathbb{N}^{6}$ with 
\begin{align*}
j_{1}+j_{2}+j_{3}+j_{4}+j_{5}+j_{7} & =k,\\
j_{1}+2j_{2}+3j_{3}+4j_{4}+5j_{5}+7j_{7} & =j,
\end{align*}
then $6j_{1}+5j_{2}+4j_{3}+3j_{4}+2j_{5}=7k-j$. To calculate the
coefficients of $[x^{j}:f^{k}]$, it suffices to list all solutions
$(j_{1},\dots,j_{5},j_{7})\in\mathbb{N}^{6}$ of the these linear
equations.

The following part of this section will deduce two explicit equalities,
by Hermite\textquoteright s criterion, for coefficients $a_{i}\in\mathbb{F}_{q}$
of a normalized PP $f(x)=x^{7}+\sum_{i=1}^{5}a_{i}x^{i}$ over $\mathbb{F}_{q}$
of characteristic $p\neq2$, on a case by case basis for all odd prime
powers $q$ with $7<q\leqslant409$.

\subsection{Case $q\equiv2$ $(\mathrm{mod}\ 7)$}

Namely, 
\[
q=p^{r}=7k+2\in\{23,37,79,107,149,163,191,233,317,331,359,373,401,3^{2},11^{2},17^{2}\},
\]
with positive integers $r$ and $k=\frac{q-2}{7}$. For an integer
$t$ coprime to $p$, let $t^{-1}$ denote a multiplicative inverse
of $t$ modulo $p$. Then 
\begin{alignat*}{3}
k & \equiv-7^{-1}\cdot2\ (\mathrm{mod}\ p),\quad & k+1 & \equiv7^{-1}\cdot5\ (\mathrm{mod}\ p),\quad & k+2 & \equiv7^{-1}\cdot4\cdot3\ (\mathrm{mod}\ p),\\
k-1 & \equiv-7^{-1}\cdot9\ (\mathrm{mod}\ p),\quad & k-2 & \equiv-7^{-1}\cdot16\ (\mathrm{mod}\ p),\quad & k-3 & \equiv-7^{-1}\cdot23\ (\mathrm{mod}\ p).
\end{alignat*}

As $\deg(f^{k+1})=7k+7=q+5<2(q-1)$, by Hermite's criterion, $[x^{q-1}:f^{k+1}]=0$.
Consider $(j_{1},\dots,j_{5},j_{7})\in\mathbb{N}^{6}$ with $j_{7}+\sum_{i=1}^{5}j_{i}=k+1$
and $7j_{7}+\sum_{i=1}^{5}ij_{i}=q-1=7k+1$. Then $6j_{1}+5j_{2}+4j_{3}+3j_{4}+2j_{5}=6$,
and $(j_{1},\dots,j_{5})$ is one of $(1,0,0,0,0)$, $(0,0,1,0,1)$,
$(0,0,0,2,0)$, $(0,0,0,0,3)$. Therefore, 
\begin{align*}
0=[x^{q-1}:f{}^{k+1}] & =\binom{k+1}{k,1}a_{1}+\binom{k+1}{k-1,1,1}a_{3}a_{5}+\binom{k+1}{k-1,2}a_{4}^{2}+\binom{k+1}{k-2,3}a_{5}^{3}\\
 & =(k+1)(a_{1}+ka_{3}a_{5}+2^{-1}ka_{4}^{2}+6^{-1}k(k-1)a_{5}^{3}).
\end{align*}
As $p\neq5$, $a_{1}=7^{-1}(2a_{3}a_{5}+a_{4}^{2}-7^{-1}\cdot3a_{5}^{3})$
for any $q$ in this case, including $q=3^{2}$ with $k=1$.

Suppose that $q\neq3^{2}$. Then $\deg(f^{k+2})=7k+14=q+12<2(q-1)$.
By Hermite's criterion, $[x^{q-1}:f{}^{k+2}]=0$. Consider $(j_{1},\dots,j_{5},j_{7})\in\mathbb{N}^{6}$
with $j_{7}+\sum_{i=1}^{5}j_{i}=k+2$ and $7j_{7}+\sum_{i=1}^{5}ij_{i}=q-1=7k+1$.
Then $6j_{1}+5j_{2}+4j_{3}+3j_{4}+2j_{5}=13$, and  
\begin{align*}
0=\  & [x^{q-1}:f{}^{k+2}]\\
=\  & \binom{k+2}{k-1,1,1,1}(a_{1}a_{2}a_{5}+a_{1}a_{3}a_{4})+\binom{k+2}{k-1,1,2}(a_{2}a_{3}^{2}+a_{2}^{2}a_{4})+\\
 & \binom{k+2}{k-2,1,1,2}(a_{2}a_{4}^{2}a_{5}+a_{1}a_{4}a_{5}^{2}+a_{2}a_{3}a_{5}^{2}+a_{3}^{2}a_{4}a_{5})+\binom{k+2}{k-2,1,3}a_{3}a_{4}^{3}+\\
 & \binom{k+2}{k-3,1,1,3}a_{3}a_{4}a_{5}^{3}+\binom{k+2}{k-3,3,2}a_{4}^{3}a_{5}^{2}+\binom{k+2}{k-3,1,4}a_{2}a_{5}^{4}+\binom{k+2}{k-4,1,5}a_{4}a_{5}^{5}\\
=\  & (k+2)(k+1)k\Bigl(a_{1}a_{2}a_{5}+a_{1}a_{3}a_{4}+2^{-1}(a_{2}a_{3}^{2}+a_{2}^{2}a_{4})\\
 & \qquad+(k-1)\cdot2^{-1}(a_{2}a_{4}^{2}a_{5}+a_{1}a_{4}a_{5}^{2}+a_{2}a_{3}a_{5}^{2}+a_{3}^{2}a_{4}a_{5}+3^{-1}a_{3}a_{4}^{3})\\
 & \qquad+(k-1)(k-2)\cdot6^{-1}(a_{3}a_{4}a_{5}^{3}+2^{-1}a_{4}^{3}a_{5}^{2}+4^{-1}a_{2}a_{5}^{4}+(k-3)\cdot20^{-1}a_{4}a_{5}^{5})\Bigr).
\end{align*}
As $p\notin\{2,3,5\}$, we have $p\nmid(k+2)(k+1)k$,  and thus  
\begin{align*}
3430(a_{1}a_{2}a_{5}+a_{1}a_{3}a_{4})+1715(a_{2}a_{3}^{2}+a_{2}^{2}a_{4})-2205(a_{2}a_{4}^{2}a_{5}+a_{1}a_{4}a_{5}^{2}+a_{2}a_{3}a_{5}^{2}+a_{3}^{2}a_{4}a_{5})\\
{}-735a_{3}a_{4}^{3}+1680a_{3}a_{4}a_{5}^{3}+840a_{4}^{3}a_{5}^{2}+420a_{2}a_{5}^{4}-276a_{4}a_{5}^{5} & =0.
\end{align*}

\subsection{Case $q\equiv3$ $(\mathrm{mod}\ 7)$}

Namely, 
\[
q=7k+3=p^{r}\in\{17,31,59,73,101,157,199,227,241,269,283,311,353,367,409\},
\]
with positive integers $r$ and $k=\frac{q-3}{7}$. Then 
\begin{alignat*}{3}
k & \equiv-7^{-1}\cdot3\ (\mathrm{mod}\ p),\quad & k+1 & \equiv7^{-1}\cdot4\ (\mathrm{mod}\ p),\quad & k+2 & \equiv7^{-1}\cdot11\ (\mathrm{mod}\ p),\\
k-1 & \equiv-7^{-1}\cdot2\cdot5\ (\mathrm{mod}\ p),\quad & k-2 & \equiv-7^{-1}\cdot17\ (\mathrm{mod}\ p),\quad & k-3 & \equiv-7^{-1}\cdot3\cdot8\ (\mathrm{mod}\ p).
\end{alignat*}

As $\deg(f^{k+1})=7k+7=q+4<2(q-1)$, by Hermite's criterion, $[x^{q-1}:f^{k+1}]=0$.
Consider $(j_{1},\dots,j_{5},j_{7})\in\mathbb{N}^{6}$ with $j_{7}+\sum_{i=1}^{5}j_{i}=k+1$
and $7j_{7}+\sum_{i=1}^{5}ij_{i}=q-1=7k+2$. Then $6j_{1}+5j_{2}+4j_{3}+3j_{4}+2j_{5}=5$,
$(j_{1},\dots,j_{5})=(0,1,0,0,0)$ or $(0,0,0,1,1)$, and 
\begin{align*}
0=[x^{q-1}:f(x)^{k+1}] & =\binom{k+1}{k,1}a_{2}+\binom{k+1}{k-1,1,1}a_{4}a_{5}\\
 & =(k+1)(a_{2}+ka_{4}a_{5}).
\end{align*}
As $p\neq2$, we have $a_{2}=7^{-1}\cdot3a_{4}a_{5}$.

Note that $\deg(f^{k+2})=7k+14=q+11<2(q-1)$. By Hermite's criterion,
$[x^{q-1}:f(x)^{k+2}]=0$. Consider $(j_{1},\dots,j_{5},j_{7})\in\mathbb{N}^{6}$
with $j_{7}+\sum_{i=1}^{5}j_{i}=k+2$ and $7j_{7}+\sum_{i=1}^{5}ij_{i}=q-1=7k+2$.
Then $6j_{1}+5j_{2}+4j_{3}+3j_{4}+2j_{5}=12$, and 
\begin{align*}
0=\  & [x^{q-1}:f(x)^{k+2}]\\
=\  & \binom{k+2}{k,2}a_{1}^{2}+\binom{k+2}{k-1,1,1,1}(a_{1}a_{3}a_{5}+a_{2}a_{3}a_{4})+\binom{k+2}{k-1,1,2}(a_{2}^{2}a_{5}+a_{1}a_{4}^{2})+\binom{k+2}{k-1,3}a_{3}^{3}+\\
 & \binom{k+2}{k-2,1,1,2}(a_{3}a_{4}^{2}a_{5}+a_{2}a_{4}a_{5}^{2})+\binom{k+2}{k-2,2,2}a_{3}^{2}a_{5}^{2}+\binom{k+2}{k-2,1,3}a_{1}a_{5}^{3}+\binom{k+2}{k-2,4}a_{4}^{4}+\\
 & \binom{k+2}{k-3,2,3}a_{4}^{2}a_{5}^{3}+\binom{k+2}{k-3,1,4}a_{3}a_{5}^{4}+\binom{k+2}{k-4,6}a_{5}^{6}\\
=\  & (k+2)(k+1)\Bigl(2^{-1}a_{1}^{2}+k(a_{1}a_{3}a_{5}+a_{2}a_{3}a_{4}+2^{-1}(a_{2}^{2}a_{5}+a_{1}a_{4}^{2})+6^{-1}a_{3}^{3}+\\
 & \qquad(k-1)(2^{-1}(a_{3}a_{4}^{2}a_{5}+a_{2}a_{4}a_{5}^{2})+4^{-1}a_{3}^{2}a_{5}^{2}+6^{-1}a_{1}a_{5}^{3}+24^{-1}a_{4}^{4}+\\
 & \qquad(k-2)(12^{-1}a_{4}^{2}a_{5}^{3}+24^{-1}a_{3}a_{5}^{4}+(k-3)720^{-1}a_{5}^{6})))\Bigr).
\end{align*}
As $p\notin\{2,11\}$, we have $p\nmid(k+2)(k+1)$, and thus 
\begin{align*}
4802a_{1}^{2}-4116(a_{1}a_{3}a_{5}+a_{2}a_{3}a_{4})-2058(a_{2}^{2}a_{5}+a_{1}a_{4}^{2})-686a_{3}^{3}+2940(a_{3}a_{4}^{2}a_{5}+a_{2}a_{4}a_{5}^{2})\\
{}+1470a_{3}^{2}a_{5}^{2}+980a_{1}a_{5}^{3}+245a_{4}^{4}-1190a_{4}^{2}a_{5}^{3}-595a_{3}a_{5}^{4}+68a_{5}^{6} & =0.
\end{align*}

\subsection{Case $q\equiv4$ $(\mathrm{mod}\ 7)$}

Namely, 
\[
q=7k+4=p^{r}\in\{11,53,67,109,137,151,179,193,263,277,347,389,3^{4},5^{2},19^{2}\},
\]
with positive integers $r$ and $k=\frac{q-4}{7}$. Then 
\begin{alignat*}{3}
k & \equiv-7^{-1}\cdot4\ (\mathrm{mod}\ p),\quad & k+1 & \equiv7^{-1}\cdot3\ (\mathrm{mod}\ p),\quad & k+2 & \equiv7^{-1}\cdot2\cdot5\ (\mathrm{mod}\ p),\\
k-1 & \equiv-7^{-1}\cdot11\ (\mathrm{mod}\ p),\quad & k-2 & \equiv-7^{-1}\cdot18\ (\mathrm{mod}\ p).\quad
\end{alignat*}

As $\deg(f^{k+1})=7k+7=q+3<2(q-1)$, by Hermite's criterion, $[x^{q-1}:f^{k+1}]=0$.
Consider $(j_{1},\dots,j_{5},j_{7})\in\mathbb{N}^{6}$ with $j_{7}+\sum_{i=1}^{5}j_{i}=k+1$
and $7j_{7}+\sum_{i=1}^{5}ij_{i}=q-1=7k+3$. So $6j_{1}+5j_{2}+4j_{3}+3j_{4}+2j_{5}=4$,
$(j_{1},\dots,j_{5})=(0,0,1,0,0)$ or $(0,0,0,0,2)$, and 
\begin{align*}
0=[x^{q-1}:f(x)^{k+1}] & =\binom{k+1}{k,1}a_{3}+\binom{k+1}{k-1,2}a_{5}^{2}\\
 & =(k+1)(a_{3}+2^{-1}ka_{5}^{2}).
\end{align*}
If $p\ne3$ (i.e. $q\neq3^{4}$ in this case), then $a_{3}=7^{-1}\cdot2a_{5}^{2}$
in $\mathbb{F}_{q}$.

Suppose $q>25$. Then $p\notin\{5,11\}$. Note that $\deg(f^{k+2})=7k+14=q+10<2(q-1)$.
By Hermite's criterion, $[x^{q-1}:f(x)^{k+2}]=0$. Consider $(j_{1},\dots,j_{5},j_{7})\in\mathbb{N}^{6}$
with $j_{7}+\sum_{i=1}^{5}j_{i}=k+2$ and $7j_{7}+\sum_{i=1}^{5}ij_{i}=q-1=7k+3$.
Then $6j_{1}+5j_{2}+4j_{3}+3j_{4}+2j_{5}=11$, and 
\begin{align*}
0=\  & [x^{q-1}:f(x)^{k+2}]\\
=\  & \binom{k+2}{k,1,1}a_{1}a_{2}+\binom{k+2}{k-1,1,1,1}(a_{1}a_{4}a_{5}+a_{2}a_{3}a_{5})+\binom{k+2}{k-1,1,2}(a_{2}a_{4}^{2}+a_{3}^{2}a_{4})+\\
 & \binom{k+2}{k-2,1,3}(a_{2}a_{5}^{3}+a_{4}^{3}a_{5})+\binom{k+2}{k-2,1,1,2}a_{3}a_{4}a_{5}^{2}+\binom{k+2}{k-3,1,4}a_{4}a_{5}^{4}\\
=\  & (k+2)(k+1)\Bigl(a_{1}a_{2}+k(a_{1}a_{4}a_{5}+a_{2}a_{3}a_{5})+k\cdot2^{-1}(a_{2}a_{4}^{2}+a_{3}^{2}a_{4})\\
 & \qquad+k(k-1)\cdot6^{-1}(a_{2}a_{5}^{3}+a_{4}^{3}a_{5})+k(k-1)\cdot2^{-1}a_{3}a_{4}a_{5}^{2}+k(k-1)(k-2)\cdot24^{-1}a_{4}a_{5}^{4}\Bigr).
\end{align*}
If $3^{4}\neq q>25$, as $p\notin\{2,3,5\}$, we have $p\nmid(k+2)(k+1)$
and 
\begin{align*}
1029a_{1}a_{2}-588(a_{1}a_{4}a_{5}+a_{2}a_{3}a_{5})-294(a_{2}a_{4}^{2}+a_{3}^{2}a_{4})\\
+154(a_{2}a_{5}^{3}+a_{4}^{3}a_{5})+462a_{3}a_{4}a_{5}^{2}\text{\textminus}99a_{4}a_{5}^{4} & =0.
\end{align*}

For $q=3^{4}=81$ with $k=11$, by Hermite's criterion, 
\begin{align*}
0 & =[x^{80}:f^{13}]=a_{2}a_{5}^{3}+a_{4}^{3}a_{5},\\
 & =[x^{80}:f^{14}]=a_{5}^{9}+a_{3}^{3}a_{4}^{2}-a_{2}a_{3}a_{4}^{3}-a_{1}a_{4}^{4}-a_{3}^{4}a_{5}+a_{1}^{2}a_{5}^{3}-a_{1}a_{3}^{3}-a_{2}^{3}a_{4}+a_{1}^{3}.
\end{align*}

\subsection{Case $q\equiv5$ $(\mathrm{mod}\ 7)$}

Namely, 
\[
q=7k+5=p^{r}\in\{19,47,61,89,103,131,173,229,257,271,313,383,397,3^{5}\},
\]
with positive integers $r$ and $k=\frac{q-5}{7}$. Then 
\begin{alignat*}{3}
k & \equiv-7^{-1}\cdot5\ (\mathrm{mod}\ p),\quad & k+1 & \equiv7^{-1}\cdot2\ (\mathrm{mod}\ p),\quad & k+2 & \equiv7^{-1}\cdot9\ (\mathrm{mod}\ p),\\
k-1 & \equiv-7^{-1}\cdot12\ (\mathrm{mod}\ p),\quad & k-2 & \equiv-7^{-1}\cdot19\ (\mathrm{mod}\ p).\quad
\end{alignat*}

 As $\deg(f^{k+1})=7k+7=q+2<2(q-1)$, by Hermite's criterion, $[x^{q-1}:f^{k+1}]=0$.
Consider $(j_{1},\dots,j_{5},j_{7})\in\mathbb{N}^{6}$ with $j_{7}+\sum_{i=1}^{5}j_{i}=k+1$
and $7j_{7}+\sum_{i=1}^{5}ij_{i}=q-1=7k+4$. Then $6j_{1}+5j_{2}+4j_{3}+3j_{4}+2j_{5}=3$,
$(j_{1},\dots,j_{5})=(0,0,0,1,0)$. So $0=[x^{q-1}:f(x)^{k+1}]=(k+1)a_{4}$.
As $p\neq2$, we have $a_{4}=0$.

Note that $\deg(f^{k+2})=7k+14=q+9<2(q-1)$. By Hermite's criterion,
$[x^{q-1}:f(x)^{k+2}]=0$. Consider $(j_{1},j_{2},j_{3},j_{5},j_{7})\in\mathbb{N}^{5}$
with $j_{1}+j_{2}+j_{3}+j_{5}+j_{7}=k+2$ and $j_{1}+2j_{2}+3j_{3}+5j_{5}+7j_{7}=q-1=7k+4$.
Then $6j_{1}+5j_{2}+4j_{3}+2j_{5}=10$, and 
\begin{align*}
0=\  & [x^{q-1}:f(x)^{k+2}]\\
=\  & \binom{k+2}{k,1,1}a_{1}a_{3}+\binom{k+2}{k,2}a_{2}^{2}+\binom{k+2}{k-1,1,2}(a_{1}a_{5}^{2}+a_{3}^{2}a_{5})+\binom{k+2}{k-2,1,3}a_{3}a_{5}^{3}+\binom{k+2}{k-3,5}a_{5}^{5}\\
=\  & 2^{-1}(k+2)(k+1)\Bigl(2a_{1}a_{3}+a_{2}^{2}+k(a_{1}a_{5}^{2}+a_{3}^{2}a_{5})+3^{-1}k(k-1)a_{3}a_{5}^{3}+60^{-1}k(k-1)(k-2)a_{5}^{5}\Bigr).
\end{align*}
If $p\neq3$ (i.e. $q=3^{5}$ in this case), then $686a_{1}a_{3}+343a_{2}^{2}-245(a_{1}a_{5}^{2}+a_{3}^{2}a_{5})+140a_{3}a_{5}^{3}-19a_{5}^{5}=0$.

For $q=3^{5}=243$ with $k=34$, by Hermite's criterion, 
\[
0=[x^{242}:f^{38}]=-a_{3}a_{5}^{10}-a_{1}a_{5}^{9}=0,
\]
So $a_{5}(a_{1}+a_{3}a_{5})=0$.

\subsection{Case $q\equiv6$ $(\mathrm{mod}\ 7)$}

Namely, 
\[
q=7k+6=p^{r}\in\{13,41,83,97,139,167,181,223,251,293,307,349,3^{3},5^{3}\},
\]
with positive integers $r$ and $k=\frac{q-6}{7}$. Then 
\[
k\equiv-7^{-1}\cdot6\ (\mathrm{mod}\ p),\quad k+1\equiv7^{-1}\ (\mathrm{mod}\ p),\quad k+2\equiv7^{-1}\cdot8\ (\mathrm{mod}\ p).
\]

As $\deg(f^{k+1})=7k+7=q+1<2(q-1)$, by Hermite's criterion, $[x^{q-1}:f^{k+1}]=0$.
Consider $(j_{1},\dots,j_{5},j_{7})\in\mathbb{N}^{6}$ with $j_{7}+\sum_{i=1}^{5}j_{i}=k+1$
and $7j_{7}+\sum_{i=1}^{5}ij_{i}=q-1=7k+5$. Then $6j_{1}+5j_{2}+4j_{3}+3j_{4}+2j_{5}=2$,
and $(j_{1},\dots,j_{5})=(0,0,0,0,1)$. So $0=[x^{q-1}:f(x)^{k+1}]=(k+1)a_{5}=7^{-1}a_{5}$,
and thus $a_{5}=0$.

Note that $\deg(f^{k+2})=7k+14=q+8<2(q-1)$. By Hermite's criterion,
$[x^{q-1}:f(x)^{k+2}]=0$. Consider $(j_{1},j_{2},j_{3},j_{4},j_{7})\in\mathbb{N}^{5}$
with $j_{7}+\sum_{i=1}^{4}j_{i}=k+2$ and $7j_{7}+\sum_{i=1}^{4}ij_{i}=q-1=7k+5$.
Then $6j_{1}+5j_{2}+4j_{3}+3j_{4}=9$ and $(j_{1},j_{2},j_{3},j_{4})=(1,0,0,1)$,
$(0,1,1,0)$, or $(0,0,0,3)$. So 
\begin{align*}
0 & =[x^{q-1}:f(x)^{k+2}]=\binom{k+2}{k,1,1}(a_{1}a_{4}+a_{2}a_{3})+\binom{k+2}{k-1,3}a_{4}^{3}\\
 & =(k+2)(k+1)\Bigl(a_{1}a_{4}+a_{2}a_{3}+6^{-1}ka_{4}^{3}\Bigr).
\end{align*}
As $p\neq2$, we have $7a_{1}a_{4}+7a_{2}a_{3}-a_{4}^{3}=0$ for all
$q$ in this case, including $q=3^{3}$.

\section{\label{sec:Result}Non-exceptional PPs of degree $7$}

Wan \citep{Wan1993padic} proved that if the value set $\{f(c):c\in\mathbb{F}_{q}\}$
of a nonzero polynomial $f\in\mathbb{F}_{q}[x]$ contains more than
$\lfloor q-\frac{q-1}{\deg(f)}\rfloor$ distinct values, then $f$
is a PP over $\mathbb{F}_{q}$. Therefore, we can define the following
SageMath function $\mathbf{isPP}(q,a_{5},a_{4},a_{3},a_{2},a_{1})$
to check whether $x^{7}+\sum_{i=1}^{5}a_{i}x^{i}$ (with all $a_{i}\in\mathbb{F}_{q}$)
is a PP over $\mathbb{F}_{q}$ or not.

\begin{lstlisting}
def isPP(q,a5,a4,a3,a2,a1):
    E = []
    for x in list(GF(q,'e'))[0:1+int(q-(q-1)/7)]:
        v = x^7+a5*x^5+a4*x^4+a3*x^3+a2*x^2+a1*x
        if v in E: return False
        else: E.append(v)
    return True
\end{lstlisting}

\subsection{When $p$ is not $2,3,7$}

By the assumptions from Section \ref{sec:Class} and equalities in
Section \ref{sec:HC}, we can write the following SageMath function
$\mathbf{PP7}(q)$ to list all PPs of degree $7$ over $\mathbb{F}_{q}$,
up to linear transformations, for any finite field $\mathbb{F}_{q}$
of characteristic $p\notin\{2,3,7\}$ with $7<q\leqslant409$. 

\begin{lstlisting}
def PP7(q):
    F = GF(q,'e'); e = F.multiplicative_generator()
    p = F.characteristic(); qmod7 = q%7; qmod3 = q%3; qmod4 = q%4
    if qmod7==1: print("No PP for q = %d" % q); return
    if p in [2,3]: print("Work only for p not 2 or 3"); return
    if q>409: print("Work only for q <= 409"); return
    if q>p: print("The minimal polynomial of e is "+str(F.modulus()))
    CK = {2:[],3:[],4:[],5:[],6:[]}; CI = {2:[],3:[],4:[],5:[],6:[]}
    for j in [2,3,4,5,6]:
        for t in range(gcd(j,q-1)): CK[j].append(e^t)
        for t in range((q-1)/gcd(j,q-1)): CI[j].append(e^t)
    for a5 in [a5 for a5 in [0,F(1),e] if qmod7!=6 or a5==0]:
        if qmod7==5: A4 = [0]
        elif a5!=0: A4 = [0]+CI[2]
        else: A4 = [0]+CK[3]
        for a4 in A4:
            if qmod7==4: A3 = [F(7)^(-1)*2*a5^2]
            elif a5==0:
                if a4!=0: A3 = [0]+CI[3]
                else: A3 = [0]+CK[4]
            else: A3 = F
            if qmod7==3: A2 = [F(7)^(-1)*3*a4*a5]
            elif a5!=0==a4: A2 = [0]+CI[2]
            else: A2 = F
            for a3 in A3:
                if qmod7==2 and p!=5:
                    A1 = [F(7)^(-1)*(2*a3*a5 + a4^2 - F(7)^(-1)*3*a5^3)]
                else: A1 = F
                for (a2,a1) in [(a2,a1) for a2 in A2 for a1 in A1 \
                if (a5==a4==a3==0)==False or (a2==0 and a1 in [0]+CK[6])\
                    or (a2 in CK[5] and a1 in [0]+CI[5])\
                if (a5==a3==0!=a4)==False or (a2 in [0]+CI[3])\
                if (a5==a4==a2==0!=a3 and qmod4==1)==False
                    or (a1 in [0]+CI[2])]:
                    if qmod7==2 and p!=5 and \
                    0!=3430*(a1*a2*a5+a1*a3*a4)+1715*(a2*a3^2+a2^2*a4)\
                    -2205*(a2*a4^2*a5+a1*a4*a5^2+a2*a3*a5^2+a3^2*a4*a5)\
                    -735*a3*a4^3+1680*a3*a4*a5^3+840*a4^3*a5^2\
                    +420*a2*a5^4-276*a4*a5^5: continue
                    if qmod7==3 and p!=11 and 0!=4802*a1^2\
                    -4116*(a1*a3*a5+a2*a3*a4)-2058*(a2^2*a5+a1*a4^2)\
                    -686*a3^3+2940*(a3*a4^2*a5+a2*a4*a5^2)+1470*a3^2*a5^2\
                    +980*a1*a5^3+245*a4^4-1190*a4^2*a5^3-595*a3*a5^4\
                    +68*a5^6: continue
                    if qmod7==4 and p!=5 and 0!=1029*a1*a2\
                    -588*(a1*a4*a5+a2*a3*a5)-294*(a2*a4^2+a3^2*a4)\
                    +154*(a2*a5^3+a4^3*a5)+462*a3*a4*a5^2-99*a4*a5^4:
                        continue
                    if qmod7==5 and p!=3 and 0!=686*a1*a3+343*a2^2\
                    -245*(a1*a5^2+a3^2*a5)+140*a3*a5^3-19*a5^5: continue
                    if qmod7==6 and 0!=7*a1*a4+7*a2*a3-a4^3: continue
                    if isPP(q,a5,a4,a3,a2,a1): print(a5,a4,a3,a2,a1)
\end{lstlisting}

We have run $\mathbf{PP7}(q)$ in SageMath for all prime powers $q$
such that $7<q\leqslant409$ and $\mathrm{gcd}(q,42)=1$. Comparing
its outputs with the exceptional polynomials given by Corollary \ref{cor:EPpnot7},
together with Lemma \ref{lem:bound}, we get the following theorem.
\begin{thm}
\label{thm:Main}Let $f$ be a non-exceptional PP over $\mathbb{F}_{q}$
with $q>7$ and $\mathrm{gcd}(q,42)=1$. Then 
\[
q\in\{11,13,17,19,23,25,31\},
\]
and $f$ is linearly related to some $x^{7}+\sum_{i=1}^{5}a_{i}x^{i}$
with $(a_{5},a_{4},a_{3},a_{2},a_{1})\in\mathbb{F}_{q}^{5}$ listed
as follows.  
\begin{alignat*}{5}
\text{For }q=11:\  & (0,0,0,5,0), & \  & (0,0,0,8,0), & \  & (0,1,0,0,4), & \  & (0,1,0,8,5), & \  & (0,1,0,9,5),\\
 & (1,0,5,0,2), & \  & (1,0,5,0,7), & \  & (1,0,5,8,6), & \  & (1,1,5,2,5), & \  & (1,2,5,9,8),\\
 & (1,4,5,8,0), & \  & (1,8,5,8,0), & \  & (1,5,5,1,5), & \  & (2,0,9,0,8), & \  & (2,0,9,0,9),\\
 & (2,0,9,4,4), & \  & (2,1,9,5,3), & \  & (2,2,9,5,8), & \  & (2,4,9,8,3), & \  & (2,8,9,5,2),\\
 & (2,8,9,7,8), & \  & (2,8,9,8,1), & \  & (2,5,9,5,2), & \  & (2,5,9,6,1), & \  & (2,5,9,8,3).\\
\text{For }q=13:\  & (0,0,0,0,2), & \  & (0,0,0,0,6), & \  & (0,0,2,0,8), & \  & (0,0,4,0,4), & \  & (0,0,8,0,3),\\
 & (0,1,0,0,2), & \  & (0,1,1,10,5), & \  & (0,1,2,1,0), & \  & (0,1,2,3,9), & \  & (0,1,8,7,11),\\
 & (0,2,1,0,8), & \  & (0,4,0,0,6), & \  & (0,4,1,7,1), & \  & (0,4,4,3,3). & \ \\
\text{For }q=17:\  & (0,1,10,0,16), & \  & (1,0,6,0,11), & \  & (1,0,7,0,0), & \  & (1,0,13,0,7), & \  & (1,0,13,0,14),\\
 & (1,0,14,0,3), & \  & (1,3,13,11,10), & \  & (1,10,3,14,11), & \  & (3,0,7,0,4), & \  & (3,0,10,0,14),\\
 & (3,0,12,0,0), & \  & (3,0,14,0,8), & \  & (3,0,15,0,2), & \  & (3,9,11,14,10), & \  & (3,9,12,14,5),\\
 & (3,9,15,14,12), & \  & (3,15,10,12,4).\\
\text{For }q=19:\  & (0,0,0,0,16), & \  & (1,0,3,14,11), & \  & (1,0,5,0,4), & \  & (1,0,7,0,11), & \  & (1,0,11,0,16),\\
 & (1,0,18,9,4), & \  & (2,0,14,0,5), & \  & (2,0,16,0,9), & \  & (2,0,17,0,5).\\
\text{For }q=23:\  & (1,1,0,4,9), & \  & (1,5,11,5,9), & \  & (1,2,6,19,21). & \ \\
\text{For }q=31:\  & (1,0,16,0,2), & \  & (1,17,25,25,29), & \  & (3,1,14,19,10). & \ \\
\text{For }q=25:\  & (0,0,0,0,e), & \  & (0,0,0,0,e^{5}), & \  & (e,0,e^{2},0,0), & \ 
\end{alignat*}
$\qquad\qquad\qquad$where $e$ is a root of $x^{2}+4x+2$ in $\mathbb{F}_{5^{2}}$,
namely, $e^{2}=e+3\in\mathbb{F}_{5^{2}}$.
\end{thm}

\subsection{When $p=3$}

Suppose $7<q=3^{r}\leqslant409$, i.e. $r\in\{2,3,4,5\}$ and $q\in\{9,27,81,243\}$.
Recall the equalities obtained in Section \ref{sec:HC} by Hermite\textquoteright s
criterion for coefficients $a_{i}\in\mathbb{F}_{q}$ of a PP $f(x)=x^{7}+\sum_{i=1}^{5}a_{i}x^{i}$
over $\mathbb{F}_{q}$.

For $q=3^{2}$, $a_{1}=7^{-1}(2a_{3}a_{5}+a_{4}^{2}-7^{-1}\cdot3a_{5}^{3})$.

For $q=3^{3}$, $0=a_{5}=7a_{1}a_{4}+7a_{2}a_{3}-a_{4}^{3}$.

For $q=3^{4}$, $0=a_{2}a_{5}^{3}+a_{4}^{3}a_{5}=a_{5}^{9}+a_{3}^{3}a_{4}^{2}-a_{2}a_{3}a_{4}^{3}-a_{1}a_{4}^{4}-a_{3}^{4}a_{5}+a_{1}^{2}a_{5}^{3}-a_{1}a_{3}^{3}-a_{2}^{3}a_{4}+a_{1}^{3}$.

For $q=3^{5}$, $0=a_{4}=a_{5}(a_{1}+a_{3}a_{5})$.

Therefore, we can modify the SageMath codes of $\mathbf{PP7}$ into
$\mathbf{PP7p3}$ to  list all PPs of degree $7$ over $\mathbb{F}_{q}$
up to linear transformations for $q\in\{3^{2},3^{3},3^{4},3^{5}\}$.

\begin{lstlisting}
def PP7p3(q):
    F = GF(q,'e'); e = F.multiplicative_generator();qmod4 = q%4
    print("The minimal polynomial of e is "+str(F.modulus()))
    CK = {2:[],3:[],4:[],5:[],6:[]}; CI = {2:[],3:[],4:[],5:[],6:[]}
    for j in [2,3,4,5,6]:
        for t in range(gcd(j,q-1)): CK[j].append(e^t)
        for t in range((q-1)/gcd(j,q-1)): CI[j].append(e^t)
    for a5 in [a5 for a5 in [0,F(1),e] if q!=3^3 or a5==0]:
        if q==3^5: A4 = [0]
        elif a5!=0: A4 = [0]+CI[2]
        else: A4 = [0]+CK[3]
        for a4 in A4:
            if a5==0:
                if a4!=0: A3 = [0]+CI[3]
                else: A3 = [0]+CK[4]
            else: A3 = F
            if a5!=0==a4: A2 = [0]+CI[2]
            else: A2 = F
            for a3 in A3:
                if q==3^2:
                    A1 = [F(7)^(-1)*(2*a3*a5 + a4^2 - F(7)^(-1)*3*a5^3)]
                else: A1 = F
                for (a2,a1) in [(a2,a1) for a2 in A2 for a1 in A1 \
                 if (a5==a4==a3==0)==False or (a2==0 and a1 in [0]+CK[6])\
                  or (a2 in CK[5] and a1 in [0]+CI[5])\
                 if (a5==a3==0!=a4)==False or (a2 in [0]+CI[3])\
                 if (a5==a4==a2==0!=a3 and qmod4==1)==False\
                  or (a1 in [0]+CI[2])]:
                    if q==3^3 and 0!=7*a1*a4+7*a2*a3-a4^3: continue
                    if q==3^4 and (0!=a2*a5^3+a4^3*a5 or 0!=\
                     a5^9+a3^3*a4^2-a2*a3*a4^3-a1*a4^4-a3^4*a5+a1^2*a5^3\
                     -a1*a3^3-a2^3*a4+a1^3): continue
                    if q==3^5 and 0!=a5*(a1+a3*a5): continue
                    if isPP(q,a5,a4,a3,a2,a1): print(a5,a4,a3,a2,a1)
\end{lstlisting}

The outputs of $\mathbf{PP7p3}(q)$ for $q\in\{3^{2},3^{3},3^{4},3^{5}\}$
and Corollary \ref{cor:EPpnot7} gives the Proposition \ref{prop:p3}.
\begin{prop}
\label{prop:p3} (1) Let $e$ be a root of $x^{2}+2x+2$ in $\mathbb{F}_{9}$
(namely, $e^{2}=e+1$). All non-exceptional PPs over $\mathbb{F}_{9}$
are linearly related to some $x^{7}+\sum_{i=1}^{5}a_{i}x^{i}$ with
$(a_{5},a_{4},a_{3},a_{2},a_{1})\in\mathbb{F}_{9}^{5}$ listed as
follows: 
\begin{alignat*}{5}
 & (0,0,e^{2},0,0), & \  & (0,1,e,1,1), & \  & (0,1,e^{2},e,1), & \  & (0,1,e^{2},2e,1), & \  & (0,1,e^{3},1,1),\\
 & (0,1,2,2,1), &  & (0,1,2e^{2},e^{3},1), &  & (0,1,2e^{2},2e^{3},1), &  & (1,0,e,0,2e), &  & (1,0,e^{3},0,2e^{3}),\\
 & (1,0,2e,0,e), &  & (1,0,2e^{3},0,e^{3}), &  & (1,0,1,0,2), &  & (1,e,e,2e^{2},1), &  & (1,e,e,1,1),\\
 & (1,e,e^{2},2,0), &  & (1,e,e^{2},1,0), &  & (1,e,e^{3},e^{2},2e), &  & (1,e,e^{3},2e^{2},2e), &  & (1,e,1,e^{2},e),\\
 & (1,e,1,2,e), &  & (1,e^{2},0,e^{2},2), &  & (1,e^{2},e,2,2e^{2}), &  & (1,e^{2},e^{2},1,e^{3}), &  & (1,e^{2},e^{3},1,e^{2}),\\
 & (1,e^{2},2,e^{2},0), &  & (1,e^{2},2e^{2},2,e), &  & (1,e^{3},e,e^{2},2e^{3}), &  & (1,e^{3},e,2e^{2},2e^{3}), &  & (1,e^{3},e^{3},e^{2},1),\\
 & (1,e^{3},e^{3},1,1), &  & (1,e^{3},2e^{2},2,0), &  & (1,e^{3},2e^{2},1,0), &  & (1,e^{3},1,2,e^{3}), &  & (1,e^{3},1,2e^{2},e^{3}),\\
 & (e,0,0,0,0), &  & (e,0,e,1,2e^{2}), &  & (e,0,e,e^{2},2e^{2}), &  & (e,0,e^{2},0,2e^{3}), &  & (e,0,2e^{3},e,2),\\
 & (e,0,2e^{3},e^{3},2), &  & (e,1,0,2,1), &  & (e,1,e^{2},2e^{2},e), &  & (e,1,e^{3},e^{3},2), &  & (e,1,e^{3},2e^{3},2),\\
 & (e,1,2,e,e^{2}), &  & (e,1,2,e+2,e^{2}), &  & (e,1,2e^{2},2e^{2},2e^{2}), &  & (e,e,0,e^{3},e^{2}), &  & (e,e,e^{2},e,2e),\\
 & (e,e,2e,2,2e^{2}), &  & (e,e,2e,1,2e^{2}), &  & (e,e,2e^{2},e,2), &  & (e,e,1,e^{2},1), &  & (e,e,1,2,1),\\
 & (e,e^{2},e,e^{2},e^{3}), &  & (e,e^{2},e^{3},2,0), &  & (e,e^{2},2,2e^{2},2e^{3}), &  & (e,e^{2},2e,e^{3},e), &  & (e,e^{2},2e,2e,e),\\
 & (e,e^{3},e^{3},e^{2},2e), &  & (e,e^{3},e^{3},2,2e), &  & (e,e^{3},2e,2e^{3},0), &  & (e,e^{3},2e^{3},e,e^{3}), &  & (e,e^{3},1,2e,2e^{3}).
\end{alignat*}

(2) All non-exceptional PPs over $\mathbb{F}_{27}$ are linearly related
to $x^{7}-x^{3}+x$. 

(3) All PPs over $\mathbb{F}_{81}$ are exceptional.
\end{prop}

\subsection{When $p=7$}

Suppose $7<q=7^{r}\leqslant409$, i.e. $r\in\{2,3\}$ and $q\in\{49,343\}$.
By Proposition \ref{prop:Eq-7} and \ref{prop:Eq-all}, each PP of
degree $7$ over $\mathbb{F}_{7^{r}}$ is linearly related to some
$f(x)=x^{7}+\sum_{i=1}^{5}a_{i}x^{i}$ with all $a_{i}\in\mathbb{F}_{7^{r}}$
satisfying the following requirements:
\begin{itemize}
\item $a_{5}\in\{0,1,e\}$, $a_{4}\in\{0,1,e,e^{2}\}$, $a_{5}a_{4}=a_{4}a_{3}=0$;
\item If $a_{5}=a_{4}=0\neq a_{3}$, then $a_{3}\in\begin{cases}
\{1,e,e^{2},e^{3}\} & \text{if }q=7^{2},\\
\{1,e\} & \text{if }q=7^{3},
\end{cases}$ and $a_{2}=0$.
\item If $a_{5}=a_{4}=a_{3}=0\neq a_{2}$, then $a_{2}=1$ and $a_{1}=0$.
\item If $a_{5}=a_{4}=a_{3}=a_{2}=0$, then $a_{1}\in\{0,1,e,e^{2},e^{3},e^{4},e^{5}\}$.\end{itemize}
\begin{prop}
Let $f(x)=x^{7}+a_{5}x^{5}+a_{3}x^{3}+a_{2}x^{2}+a_{1}x$ be a PP
over $\mathbb{F}_{7^{r}}$ with $r\in\{2,3\}$ and all $a_{i}\in\mathbb{F}_{7^{r}}$
and $a_{5}\ne0$. Then $a_{1}=3(a_{3}a_{5}+a_{3}^{2}a_{5}^{-1})$.\end{prop}
\begin{proof}
By Hermite\textquoteright s Criterion, we have 
\begin{alignat*}{3}
0 & =[x^{48}:f^{10}] &  & =3a_{3}a_{5}^{9}+3a_{3}^{2}a_{5}^{7}-a_{1}a_{5}^{8}, & \quad & \text{if }r=2,\\
0 & =[x^{342}:f^{66}] &  & =-a_{3}a_{5}^{58}-a_{3}^{2}a_{5}^{56}-2a_{1}a_{5}^{57}, &  & \text{if }r=3,
\end{alignat*}
so $a_{1}=3(a_{3}a_{5}+a_{3}^{2}a_{5}^{-1})$ in both cases.
\end{proof}
Therefore, for $r\in\{2,3\}$, the following SageMath codes runs
to list all the PPs over $\mathbb{F}_{7^{r}}$ up to linear transformatons.

\begin{lstlisting}
for q in [7^2,7^3]:
    F=GF(q,'e'); e=F.multiplicative_generator()
    print("For q = %s, e is a root of" % (q)),; print(F.modulus())
    for a5 in [0,1,e]:
        for a4 in [a4 for a4 in [0,1,e,e^2] if a4*a5==0]:
            if a4!=0: A3=[0]
            elif a5==0 and q==7^2: A3=[0,1,e,e^2,e^3]
            elif a5==0 and q==7^3: A3=[0,1,e]
            else: A3=F
            for a3 in A3:
                if a5==a4==0: A2=[a2 for a2 in [0,1] if a2*a3==0]
                else: A2=F
                for a2 in A2:
                    if a5!=0: A1=[3*(a3*a5+a3^2/a5)]
                    elif a4==a3==0:
                        if a2==0: A1=[0,1,e,e^2,e^3,e^4,e^5]
                        else: A1=[0]
                    else: A1=F
                    for a1 in A1:
                        if isPP(q,a5,a4,a3,a2,a1): print(a5,a4,a3,a2,a1)
\end{lstlisting}

It is easy to check which polynomials in the outputs are exceptional
by Lemma \ref{lem:EP7p7}. Then we can reword the outputs as the following
two propositions.
\begin{prop}
\label{prop:p7} Fix a root $e$ of $x^{2}+6x+3$ in $\mathbb{F}_{7^{2}}$.
All non-exceptional PPs over $\mathbb{F}_{7^{2}}$ are linearly related
to 
\[
x^{7}+ex^{5}+e^{18}x^{3}+e^{35}x.
\]
Each exceptional polynomial over $\mathbb{F}_{7^{2}}$ is linearly
related to exactly one of the following:\begin{itemize}

\item $x^{7}+ax$ with $a\in\{0,e,e^{2},e^{3},e^{4},e^{5}\}$;

\item $x^{7}+x^{4}+2x$, $x^{7}+e^{2}x^{4}+2e^{4}x$;

\item $x^{7}+ex^{5}+5e^{2}x^{3}+6e^{3}x$.\end{itemize}
\end{prop}

\begin{prop}
Fix a root $e$ of $x^{3}+6x^{2}+4$ in $\mathbb{F}_{7^{3}}$. Each
PP over $\mathbb{F}_{7^{3}}$ is exceptional, and linearly related
to exactly one of the following polynomials:\begin{itemize}

\item $x^{7}+bx$ with $b\in\{0,1,e,e^{2},e^{4},e^{5}\}$;

\item $x^{7}+ex^{4}+2e^{2}x$, $x^{7}+e^{2}x^{4}+2e^{4}x$;

\item $x^{7}+ex^{5}+5e^{2}x^{3}+6e^{3}x$.\end{itemize}
\end{prop}

\end{document}